\documentclass{amsart}

\usepackage{mathptmx}

\usepackage{amssymb}
\usepackage{amsfonts}
\usepackage{amsmath}
\usepackage{graphicx}
\usepackage{shadow}
\usepackage{color}
\usepackage[all]{xy}
\usepackage{tikz}
\usepackage{verbatim}

\newtheorem{thm}{Theorem}[section]
\newtheorem{corollary}[thm]{Corollary}
\newtheorem{lemma}[thm]{Lemma}
\newtheorem{proposition}[thm]{Proposition}

\theoremstyle{definition}
\newtheorem{definition}[thm]{Definition}
\newtheorem{example}[thm]{Example}

\theoremstyle{remark}

\newcommand{\bc}{\color{blue}}
\newcommand{\bs}{\bigskip}
\newcommand{\st}{\star}

\newcommand{\mc}{\mathcal}

\DeclareMathOperator{\htt}{ht}

\DeclareMathOperator{\Max}{Max}

\begin{document}

\title[$\st$-reductions of ideals and P$v$MDs]{\Large {$\st$-reductions of ideals and Pr\"ufer $v$-multiplication domains}}

\thanks{$^{(\star)}$ Supported by King Fahd University of Petroleum \& Minerals under DSR Research Grant \#: RG1328.}

\author[E. Houston]{E. Houston}
\address{Department of Mathematics and Statistics, University of North Carolina at Charlotte, Charlotte, NC 28223, USA}
\email{eghousto@uncc.edu}

\author[S. Kabbaj]{S. Kabbaj $^{(\star)}$}
\address{Department of Mathematics and Statistics, King Fahd University of Petroleum \& Minerals, Dhahran 31261, KSA}
\email{kabbaj@kfupm.edu.sa}

\author[A. Mimouni]{A. Mimouni $^{(\star)}$}
\address{Department of Mathematics and Statistics, King Fahd University of Petroleum \& Minerals, Dhahran 31261, KSA}
\email{amimouni@kfupm.edu.sa}

\date{\today}



\begin{abstract}
Let $R$ be a commutative ring and $I$ an ideal of $R$. An ideal $J\subseteq I$ is a reduction of $I$ if $JI^{n}=I^{n+1}$ for some positive integer $n$. The ring $R$ has the (finite) basic ideal property if (finitely generated) ideals  of $R$ do not have proper reductions. Hays characterized (one-dimensional) Pr\"ufer domains as domains with the finite basic ideal property (basic ideal property).  We extend Hays' results to Pr\"ufer $v$-multiplication domains by replacing ``basic'' with ``$w$-basic,'' where $w$ is a particular star operation.  We also investigate relations among $\st$-basic properties for certain star operations $\st$.
\end{abstract}
\maketitle


\section*{introduction}

Throughout, all rings considered are commutative with identity. Let $R$ be a ring and $I$ an ideal of $R$. An ideal $J\subseteq I$ is a \emph{reduction} of $I$ if $JI^{n}=I^{n+1}$ for some positive integer $n$  \cite{NR}. An ideal that has no reduction other than itself is called a \emph{basic} ideal \cite{H1}. The notion of reduction was introduced by Northcott and Rees, who stated: ``First, it defines a relationship between two ideals which is preserved under homomorphisms and ring extensions; secondly, what we may term the reduction process gets rid of superfluous elements of an ideal without disturbing the algebraic multiplicities associated with it" \cite{NR}. For both early and recent developments on reduction theory, we refer the reader to \cite{HRR,H1,H2,HS2,NR,RS,SK,Wang}.

In \cite{H1,H2}, Hays investigated reductions of ideals in commutative rings with a particular focus on Pr\"ufer domains. He studied the notion of basic ideal and examined domains subject to the basic ideal property (i.e., every ideal is basic). This class is shown to be strictly contained in the class of Pr\"ufer domains (domains in which every nonzero finitely generated ideal is invertible); and  a new characterization for Pr\"ufer domains is provided; namely, a domain is Pr\"ufer if and only if it has the finite basic ideal property (i.e., every finitely generated ideal is basic) \cite[Theorem 6.5]{H1}. The second main result of these two papers characterizes domains with the (full) basic ideal property as one-dimensional Pr\"ufer domains (\cite[Theorem 6.1]{H1} combined with \cite[Theorem 10]{H2}). Our primary goal is to extend Hays' results to Pr\"ufer $v$-multiplication domains (P$v$MDs).

Let $R$ be a domain and $I$ a nonzero fractional ideal of $R$. The $v$- and $t$-closures of $I$ are defined, respectively, by $I_v:=(I^{-1})^{-1}$ and $I_t:=\cup J_v$, where $J$ ranges over the set of finitely generated subideals of $I$. Recall that $I$ is a $t$-ideal if $I_t=I$ and a $t$-finite (or $v$-finite) ideal if there exists a finitely generated fractional ideal $J$ of $R$ such that $I=J_{t}=J_{v}$; and $R$ is called a Pr\"ufer $v$-multiplication domain (P$v$MD) if the set of its $t$-finite $t$-ideals forms a group under ideal $t$-multiplication ($(I,J)\mapsto (IJ)_t$). A useful characterization is that $R$ is a P$v$MD if and only if each localization at a maximal $t$-ideal is a valuation domain \cite[Theorem 5]{gr}.  The class of P$v$MDs strictly contains the classes of factorial and Pr\"ufer domains. The $t$-operation is nowadays a cornerstone of multiplicative ideal theory and has been investigated thoroughly by many commutative algebraists since the 1980's.

For the convenience of the reader, the following figure displays a diagram of implications summarizing the relations among many well-studied classes of domains, putting P$v$MDs in perspective.
In the diagram, classes on top become the classes directly underneath by means of replacing the definitions with a corresponding $t$-version.  For example, a GCD-domain is a domain in which $I_t$ is principal for each nonzero finitely generated ideal $I$, and a P$v$MD is a domain in which each nonzero finitely generated ideal is $t$-invertible.


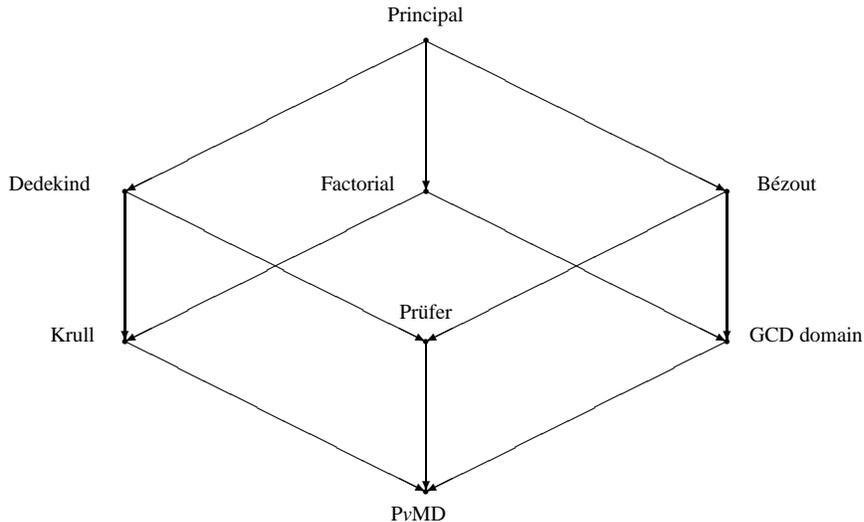
\begin{figure}[h!]
\centering
\[\setlength{\unitlength}{1mm}
\begin{picture}(80,65)(0,-90)
\put(40,-30){\vector(0,-1){20}}
\put(40,-30){\vector(2,-1){40}}
\put(40,-30){\vector(-2,-1){40}}
\put(0,-50){\vector(0,-1){20}}
\put(0,-50){\vector(2,-1){40}}
\put(80,-50){\vector(0,-1){20}}
\put(80,-50){\vector(-2,-1){40}}
\put(40,-50){\vector(2,-1){40}}
\put(40,-50){\vector(-2,-1){40}}
\put(0,-70){\vector(2,-1){40}}
\put(80,-70){\vector(-2,-1){40}}
\put(40,-70){\vector(0,-1){20}}

\put(40,-30){\circle*{.7}} \put(40,-28){\makebox(0,0)[b]{\footnotesize Principal}}
\put(0,-50){\circle*{.7}} \put(-10,-50){\makebox(0,0)[b]{\footnotesize Dedekind}}
\put(40,-50){\circle*{.7}}\put(31,-50){\makebox(0,0)[b]{\footnotesize Factorial}}
\put(80,-50){\circle*{.7}} \put(88,-50){\makebox(0,0)[b]{\footnotesize B\'ezout}}
\put(0,-70){\circle*{.7}}\put(-7,-70){\makebox(0,0)[b]{\footnotesize Krull}}
\put(80,-70){\circle*{.7}}\put(83,-69){\makebox(0,0)[l]{\footnotesize GCD domain}}
\put(40,-70){\circle*{.7}}\put(40,-67){\makebox(0,0)[b]{\footnotesize Pr\"ufer}}
 \put(40,-90){\circle*{.7}} \put(39,-92){\makebox(0,0)[t]{\footnotesize P$v$MD}}

\end{picture}\]
\caption{P$v$MDs in perspective}\label{D1}
\end{figure}

The $t$- and $v$-operations are examples of star operations (defined below).  We also require the $w$-operation: for a nonzero ideal fractional $I$ of a domain $R$, $I_w=\bigcup (I:J)$, where the union is taken over all finitely generated ideals $J$ of $R$ that satisfy $J_v=R$; equivalently, $I_w=\bigcap IR_M$, where the intersection is taken over the set of maximal $t$-ideals of $R$. It follows that for each $I$ and maximal $t$-ideal $M$, we have $I_wR_M=IR_M$. (This can be done in greater generality--see \cite{ac}.)  In the diagram above, one can replace ``$t$'' by ``$w$'' to go from top to bottom.

In Section~\ref{s:starbasic} we discuss the notion of $\st$-basic ideals and prove that a domain with the finite $\st$-basic ideal property ($\st$-basic ideal property) must be integrally closed (completely integrally closed).  We also observe that a domain has the $v$-basic ideal property if and only if it is completely integrally closed.  Section~\ref{s:characterizations} is devoted to generalizing Hays' results; we show that a domain has the finite $w$-basic ideal property ($w$-basic ideal property) if and only if it is a P$v$MD (of $t$-dimension one).  In Section~\ref{s:examples}, we present a diagram of implications among domains having various $\st$-basic properties and give examples showing that most of the implications are not reversible.  For example, a domain with the $w$-basic ideal property must also have the $t$-basic ideal property and a $v$-domain must have the finite $v$-basic ideal property, but neither implication is reversible.

Notation is standard, as in \cite{g}.  In particluar, for a domain $D$ with quotient field $K$ and submodules $A,B$ of $K$, we use $(A:B)$ to denote the $D$-module $\{x \in K \mid xB \subseteq A\}$. \bs


\section{$\st$-basic ideals} \label{s:starbasic}

Let $R$ be a domain with quotient field $K$, and let $\mc F(R)$ denote the set of nonzero fractional ideals of $R$.  A map $\st: \mc F(R)  \to \mc F(R)$, $I \mapsto I^{\st}$, is said to be a \emph{star operation on $R$} if the following conditions hold for every nonzero $a \in K$ and $I,J \in \mc F(R)$: (1) $(aI)^{\st} = aI^{\st}$ and $R^{\st}=R$; (2) $I \subseteq I^{\st}$ and $I \subseteq J$ implies $I^{\st} \subseteq J^{\st}$; and (3) $I^{\st \st}=I^{\st}$. It is common to denote the trivial star operation ($I \mapsto I$) by ``$d$.''

\begin{definition}
Let $R$ be an integral domain and $\star$ a star operation on $R$. Let $I$ be a nonzero ideal of $R$.
\begin{enumerate}
\item An ideal $J\subseteq I$ is a \emph{$\star$-reduction} of $I$ if $(JI^{n})^{\star}=(I^{n+1})^{\star}$ for some integer $n\geq0$. The ideal $J$ is a \emph{trivial $\star$-reduction} of $I$ if $J^{\star}=I^{\star}$.
\item $I$ is \emph{$\star$-basic} if it has no $\star$-reduction other than the trivial $\star$-reduction(s).

\item $R$ has the $\star$-basic ideal property if every nonzero ideal (or, equivalently, every $\star$-ideal) of $R$ is $\star$-basic.

\item $R$ has the finite $\star$-basic ideal property if every nonzero finitely generated ideal (or, equivalently, every $\star$-finite ideal) of $R$ is $\star$-basic.
\end{enumerate}
\end{definition}

It is clear that $\st$-reductions can be extended to fractional ideals; in particular, if $R$ has the $\st$-basic ideal property, then every nonzero fractional ideal of $R$ is $\st$-basic.

It is easy to see that if $\st_1 \le \st_2$ are star operations on a domain $R$ (meaning that $I^{\st_1} \subseteq I^{\st_2}$ for each $I \in \mc F(R)$), then each $\st_1$ 
reduction of an ideal is also a $\st_2$-reduction.  The converse is false.  In particular, a $t$-reduction may not be a ($d$-)reduction. For a very simple example, 
let $R=k[x,y]$ be a polynomial ring in two indeterminates over a field $k$, and let $M=(x,y)$.  Then $M$ is basic, i.e., $M$ has no reductions other than itself \cite[Theorem 2.3]{H1}. 
On the other hand, $M_t=R$ (see, e.g., \cite[Exercise 1, p. 102]{Ka}), from which it follows that any power of $M$ is a (trivial) $t$-reduction of $M$. (We give a ``better'' example following Proposition~\ref{p:cic} below.)

\begin{lemma} \label{l:invertible}
In an integral domain $R$, $\star$-invertible ideals and $\star$-idempotent ideals are $\star$-basic.
\end{lemma}
\begin{proof}
Let $J\subseteq I$ be a $\st$-reduction  of the ideal $I$ of $R$, so that $(JI^{n})^{\star}=(I^{n+1})^{\star}$ for some positive integer $n$. If $I$ is $\st$-invertible, then multiplication by $(I^{-1})^n$ and taking $\st$-closures immediately yields $J^{\star}=I^{\star}$. Next, assume that $(I^{2})^{\star}=I^{\star}$. Then $I^{\star}=(I^{n+1})^{\star}=(JI^{n})^{\star}\subseteq J^{\star}\subseteq I^{\star}$ so that, again, $J^{\star}=I^{\star}$, as desired.
\end{proof}

\begin{lemma} \label{l:starbasic} {\rm{(cf. \cite[Lemma 6.4]{H1})}} Let $\st$ be a star operation on a domain $R$. If $R$ has the finite $\st$-basic ideal property, then $R$ is integrally closed.
\end{lemma}
\begin{proof} Let $x,y \in R$ be such that $x/y$ is integral over $R$. As in the proof of \cite[Lemma 6.4]{H1}, $(y)$ is a reduction of $(x,y)$.  We then have $x \in (x,y)^{\st}=(y)^{\st}=(y)$, whence $x/y \in R$.
\end{proof}

Recall that a domain $R$ is said to be \emph{completely integrally closed} if every nonzero ideal of $R$ is $v$-invertible.

\begin{proposition} \label{p:cic}
Let $\star$ be a star operation on an integral domain $R$.
\begin{enumerate}
\item If $R$ has the $\star$-basic ideal property, then $R$ is completely integrally closed.
\item $R$ has the $v$-basic ideal property if and only if $R$ is completely integrally closed.
\end{enumerate}
\end{proposition}
\begin{proof}
(1) Assume $R$ has the $\star$-basic ideal property. Let $I$ be a nonzero ideal of $R$ and set $J:=II^{-1}$. It is well known that $J^{-1}=(J:J)$, and hence $J^{-1}$ is a ring.  Now, let $0\not=a\in J$ and set $A:=aJ^{-1}$ and $B:=aR$. Clearly, $A$ and $B$ are $v$-ideals of $R$ with $B\subseteq A$ and $BA=A^{2}$. That is, $B$ is a reduction (and, a fortiori, a $\star$-reduction) of $A$. By the $\star$-basic hypothesis, $aJ^{-1}=A^{\st}=B^{\st}=aR$, whence $R=J^{-1}$. Therefore, $(II^{-1})_{v}=J_{v}=R$, as desired.

(2) The ``only if" assertion is a special case of (1), and the converse is handled by Lemma~\ref{l:invertible}.
\end{proof}

Next, we give an example of $t$-ideals $I,J$ in a Noetherian domain $R$ such that $J$ is a $t$-reduction, but not a $d$-reduction, of $I$.  Since the $v$- and $t$-operations coincide in any Noetherian domain, such an $R$ cannot be (completely) integrally closed by Proposition~\ref{p:cic}.

\begin{example}\label{d:exa2} Again let $k$ be a field and $x,y$ indeterminants over $k$.  Let $T=k[x,y]=k+M$, where $M=(x,y)T$. Now let $R=k+M^2$.  Observe that $R$ is Noetherian (see, e.g., \cite{bsw}). As in the discussion preceding Lemma~\ref{l:invertible}, as an ideal of $T$, $M$ has no reductions other than itself.  In particular, $M^2$ is not a reduction of $M$ in $T$, and it follows easily that $M^2$ is not a reduction of (the fractional ideal) $M$ in $R$. However, we claim that $M^2$ is a nontrivial $t$-reduction of $M$.  To verify this, proceed as follows.  First, we have $(T:M)=T$ (as before).  It follows that $M \subseteq M^{-1}$ ($= (R:M)$) $\subseteq T$.  On the other hand, if $f \in T$ satisfies $fM \subseteq R$, then, writing $f=a+m$ with $a \in k$ and $m \in M$, we immediately obtain that $aM \subseteq R$, whence $a=0$, i.e., $f \in M$.  Thus $M^{-1}=M$, whence also $M_t=M_v=M$.  However, $(R:T)=M^2$, whence $(M^2)^{-1}=((R:M):M)=(M:M)=T$ and then $(M^2)_t=(M^2)_v=(R:T)=M^2$.  A similar argument yields $(M^n)_t=M^2$ for $n \ge 2$.  Hence $M^2 =(M^3)_t=(M^2M)_t$, and therefore $J:=M^2$ is a nontrivial $t$-reduction of $I:=M$, as claimed.  (To obtain an example involving integral ideals, replace $M$ by $xM$ and $M^2$ by $xM^2$.) \qed
\end{example}

We recall that a domain $R$ is a \emph{$v$-domain} if each nonzero finitely generated ideal of $R$ is $v$-invertible.
From Lemma~\ref{l:invertible}, the following is immediate:

\begin{proposition} \label{p:vdomain}  A $v$-domain has the finite $v$-basic ideal property. \qed
\end{proposition}

Now recall that to any star operation $\st$ on a domain $R$, we may define an associated star operation $\st_f$ by setting, for each $I \in \mc F(R)$, $I^{\st_f}=\bigcup J^{\st}$, the union being taken over all finitely generated subideals $J$ of $I$; the star operation $\st$ has \emph{finite type} if $\st=\st_f$.  Note that $v_f=t$.  If $\st$ is a finite-type star operation on a domain $R$, then minimal primes of $\st$-ideals are themselves $\st$-ideals and each $\st$-ideal is contained in a maximal $\st$-ideal.

\begin{lemma}\label{l:finitestarbasic}
Let $\st$ be a star operation of finite type on an integral domain $R$.  If $I$ is a finitely generated ideal of $R$ and $J$ is a $\st$-reduction of $I$, then there is a finitely generated ideal $K\subseteq J$ such that $K$ is a $\st$-reduction of $I$.
\end{lemma}
\begin{proof}
Suppose that $I$ is a finitely generated ideal of $R$ and that $(JI^{n})^{\st}=(I^{n+1})^{\st}$ for some ideal $J \subseteq I$ and some positive integer $n$. Suppose that $I^{n+1}$ is generated by $b_{1}, ..., b_{r}$ in $R$. Since  $b_{i}\in (JI^{n})^{\st}$, there is a finitely generated subideal $K_i$ of $J$ such that $b_i \in (K_iI^n)^{\st}$.  For $K = \sum_{i=1}^r K_i$, we then have $I^{n+1} \subseteq (KI^n)^{\st}$, as desired.
\end{proof}

\begin{proposition} \label{p:vimpliest} If a domain $R$ has the finite $\st$-basic ideal property, then $R$ also has the finite $\st_f$-basic ideal property.  In particular, if $R$ has the finite $v$-basic ideal property, then $R$ also has the finite $t$-basic ideal property.
\end{proposition}
\begin{proof} Let $R$ be a domain with the $\st$-basic ideal property.  Let $I$ be a finitely generated ideal of $R$, and let $J$ be a $\st_f$-reduction of $I$.  By Lemma~\ref{l:finitestarbasic} we may assume that $J$ is finitely generated.  Since $J$ is also a $\st$-reduction of $I$, we have $J^{\st_f}=J^{\st}=I^{\st}=I^{\st_f}$.  Hence $R$ has the $\st_f$-basic ideal property.
\end{proof}

\begin{corollary} \label{c:vdomain} A $v$-domain has the finite $t$-basic ideal property. \qed
\end{corollary}


\section{Characterizations} \label{s:characterizations}

We begin with an analogue of Hays' first result that a domain is a Pr\"ufer domain if and only if it has the finite basic ideal property.  We shall need a result of Kang \cite[Theorem 3.5]{Kg2} that characterizes P$v$MDs as integrally closed domains in which the $t$- and $w$-operations coincide.  We denote the set of maximal $t$-ideals of a domain $R$ by $\Max_t(R)$.

\begin{thm} \label{t:finitewbasic} {\rm{(cf. \cite[Theorem 6.5]{H1})}} A domain $R$ is a P$v$MD if and only if it has the finite $w$-basic ideal property.
\end{thm}
\begin{proof} If $R$ is a P$v$MD, then, as mentioned above, the $t$- and $w$-operations coincide, and $R$ has the finite $w$-basic ideal property by Corollary~\ref{c:vdomain}.

Now assume that $R$ has the finite $w$-basic ideal property.  Then $R$ is integrally closed by Lemma~\ref{l:starbasic}. Let $M \in \Max_t(R)$, and let $a,b \in M$.  Since $(a^2,b^2)$ is a reduction of $(a,b)^2$, we have $(a^2,b^2)_w=((a,b)^2)_w$ and hence (as mentioned in the introduction) $(a^2,b^2)R_M= (a,b)^2R_M$.  Thus $R_M$ is a valuation domain \cite[Theorem 24.3(4)]{g}. Therefore, $R$ is a P$v$MD.
\end{proof}

Hays proved that, in a Pr\"ufer domain, the definition of a reduction can be restricted; namely,  $J\subseteq I$ is a reduction if and only if $JI=I^{2}$ \cite[Proposition 1]{H2}. The next lemma establishes a similar property for $t$-reductions and also shows that this notion is local in the class of P$v$MDs. It is useful to note if $J$ is a $t$-reduction of an ideal $I$, then a prime $t$-ideal of $R$ contains $I$ if and only if it contains $J$.  We shall also need the fact (which follows easily from \cite[Lemma 4]{z} and is stated explicitly in \cite[Lemma 3.4]{Kg2}), that if $I$ is a nonzero ideal of a domain $R$ and $S$ is a multiplicatively closed subset of $R$, then $(I_tR_S)_{t_{R_S}}=(IR_S)_{t_{R_S}}$.
\begin{lemma}\label{l:square} Let $R$ be a P$v$MD and $J\subseteq I$ nonzero ideals of $R$. Then, the following assertions are equivalent:
\begin{enumerate}
\item $J$ is a $t$-reduction of $I$;
\item $JR_MIR_{M}=(IR_{M})^2$ for each $M\in\Max_{t}(R)$;
\item $(JI)_{t}=(I^{2})_{t}$.
\end{enumerate}
\end{lemma}
\begin{proof}
(1) $\Rightarrow$ (2) Assume that $J$ is a $t$-reduction of $I$, so that $(JI^n)_t=(I^{n+1})_t$ for some positive integer $n$, and let $M\in\Max_{t}(R)$. Since $R_M$ is a valuation domain, the $t$-operation is trivial on $R_M$ ($t_{R_M}=d_{R_M}$).  Using this and the remarks above, we have $$I^{n+1}R_M = ((I^{n+1})_tR_M)_{t_{R_M}} = ((JI^n)_tR_M)_{t_{R_M}}= JI^nR_M.$$ Hence $JR_M$ is a $t$-reduction of $IR_M$ in $R_M$, and so $JR_MIR_M=(IR_M)^2$ by \cite[Proposition 1]{H2}.

(2) $\Rightarrow$ (3) By \cite[Theorem 3.5]{Kg2}, we have
$$(JI)_{t}=\bigcap_{M\in\Max_{t}(R)} JIR_M=\bigcap_{M\in\Max_{t}(R)} (I^2R_M)=(I^{2})_{t}.$$

(3) $\Rightarrow$ (1) is trivial.
\end{proof}
\begin{lemma} \label{l:haysprop9} {\rm{(cf. \cite[Lemma 9]{H2})}} Let $x$ be a nonzero element of a P$v$MD $R$, let $P$ be a minimal prime of $xR$, and let $I=xR_P \cap R$.  Then \begin{enumerate}
\item $I$ is a $w$-ideal of $R$,
\item $xR+I^2$ is a $w$-reduction of $I$, and
\item if $I$ is $w$-basic, then $P \in \Max_t(R)$.
\end{enumerate}
\end{lemma}
\begin{proof} (1)\,-\,(2) Let $M$ be a maximal $t$-ideal of $R$ containing $P$.  Then $I_w \subseteq IR_M \cap R \subseteq IR_P \cap R=I$, proving (1).  We next claim that $IR_M=IR_P \cap  R_M$.  To see this, suppose that $y \in IR_P \cap R_M$.  Then we may write $y=a/s=b/t$ with $a \in I$, $b \in R$, $s \in R \setminus P$ and $t \in R \setminus M$.  We then have $b=at/s \in IR_P \cap R=I$, and hence $y =b/t  \in IR_M$, as desired. Now, for $s \in R \setminus P$ and $a \in I$ (using the fact that $R_M$ is a valuation domain), it is clear that $a/s \in IR_P \cap R_M=IR_M$.  If we also have $b \in I$, then, writing $b=x/s'$ with $s' \in R \setminus P$, we obtain $ab=(a/s')x \in xIR_M$.  Thus $I^2R_M = xIR_M$, and it follows that $(xR+I^2)IR_M=xIR_M=I^2R_M$. (In particular, $I^2R_M \subseteq xR_M$; we use this below.)  Since $I$ is $P$-primary, we also have $(xR+I^2)IR_N=I^2R_N$ for $N \in \Max_t(R)$ with $N \nsupseteq P$.  Therefore, $((xR+I^2)I)_w=(I^2)_w$, and so $xR+I^2$ is a $w$-reduction of $I$.

(3) Assume that $I$ is $w$-basic; then $(xR+I^2)_w=I_w$ by (2). Suppose that $M \in \Max_t(R)$ properly contains $P$, and choose $y \in M \setminus P$.  Then $P$ is minimal over $yx$, and $I=yxR_P \cap R$.  Hence, as above, we have (using the parenthetical ``in particular'' comment above) $xR_M \in IR_M=(yxR+I^2)R_M\subseteq yxR_M$, a contradiction. Therefore, $P \in \Max_t(R)$.
\end{proof}

\begin{thm} \label{t:basic} A domain $R$ has the $w$-basic ideal property if and only if $R$ is a P$v$MD of $t$-dimension 1.
\end{thm}
\begin{proof} Let $R$ be a P$v$MD with $t$-$\dim(R)=1$, and let $J \subseteq I$ be a nonzero ideals of $R$ with $(JI)_w = (I^2)_w$.  Let $M$ be a maximal $t$-ideal of $R$.  Then $JIR_M=I^2R_M$.  We wish to show that $JR_M=IR_M$, and for this we may as well assume that $I \subseteq M$ and $IR_M$ is not invertible.  Since $R_M$ is a valuation domain, we then have $IR_M=IMR_M$, and since $R_M$ is also one-dimensional, \cite[Proposition 2.1]{FHP} yields $IR_M(R_M:IR_M)=MR_M$.  Hence multiplying both sides of the equation $JIR_M=I^2R_M$ by $(R_M:IR_M)$ yields $JR_M \supseteq JMR_M=IMR_M=IR_M$.  We then obtain $J_w=I_w$.  Therefore, by Lemma~\ref{l:square}, $R$ has the $w$-basic ideal property.

Conversely, suppose that $R$ has the $w$-basic ideal property.  Then $R$ is a P$v$MD by Theorem~\ref{t:finitewbasic}. Let $M$ be a maximal $t$-ideal of $R$, let $Q$ be a nonzero prime of $R$ contained in $M$, let $x$ be a nonzero element of $Q$, and shrink $Q$ to a prime $P$ minimal over $x$. Then, since $I:=xR_P \cap R$ is $w$-basic by hypothesis, Lemma~\ref{l:haysprop9} yields $P=Q=M$.  Therefore, $\htt M=1$, as desired.
\end{proof}


\section{Examples.} \label{s:examples}

Consider the following diagram of implications involving various $\st$-basic properties.


\begin{figure}[h!]
\begin{tikzpicture} [xscale = 1.3,yscale = 1.3]
\draw [fill] (0,-.03) circle [radius = .03];\draw  [->] (0,0) -- (0,-.97);
\node [right] at (0,-.02) {\footnotesize{Krull}};
\draw [fill] (0,-1) circle [radius = .03]; \draw [->] (0,-1) -- (0,-1.97);
\node[right] at (0,-.99) {\footnotesize{\bc $w$-basic} = P$v$MD + $t$-dim 1};
\draw [fill] (0,-2) circle [radius = .03]; \draw [->] (0,-2) -- (0,-2.97);
\node[right] at (0,-1.99) {\footnotesize{\bc finite $w$-basic} = P$v$MD};
\draw [fill] (0,-3) circle [radius = .03]; \draw [->] (0,-3) -- (0,-3.97);
\node[right] at (0,-2.99) {\footnotesize{$v$-domain}};
\draw [fill] (0,-4) circle [radius = .03]; \draw [->] (0,-4) -- (0,-4.97);
\node[right] at (0,-3.99) {\footnotesize{\bc finite $v$-basic}};
\draw [fill] (0,-5) circle [radius = .03]; \draw [->] (0,-5) -- (0,-5.97);
\node[right] at (0,-4.99) {\footnotesize{\bc finite $t$-basic}};
\node[right] at (0,-5.99) {\footnotesize{integrally closed}};
\draw [fill] (0,-6) circle [radius = .03];
\draw [->] (0,-.99) -- (-1.5,-1.37); \draw[fill] (-1.5,-1.4) circle [radius = .03];
\draw[->] (-1.5,-1.4) -- (-1.5,-2.47); \draw[fill] (-1.5,-2.5) circle [radius = .03];
\draw[->] (-1.5,-2.5) -- (0,-2.99);
\node[left] at (-1.5,-1.37) {\footnotesize{\bc $t$-basic}};
\node[left] at (-1.5,-2.47) {\footnotesize{ \bc $v$-basic} = completely};
\node[left] at (-1.5,-2.77) {\footnotesize{integrally closed}};
\node[right] at (0,-.5) {\footnotesize{(1)}};
\node[right] at (0,-1.5) {\footnotesize{(2)}};
\node[right] at (0,-2.5) {\footnotesize{(3)}};
\node[right] at (0,-3.5) {\footnotesize{(4)}};
\node[right] at (0,-4.5) {\footnotesize{(5)}};
\node[right] at (0,-5.5) {\footnotesize{(6)}};
\node[left] at (-1.46,-1.95) {\footnotesize{(8)}};
\node[above right] at (-.95,-1.25) {\footnotesize{(7)}};
\node[below] at (-.75,-2.7) {\footnotesize{(9)}};

\end{tikzpicture}
\caption{$\star$-basic properties in perspective}\label{D2}
\end{figure}
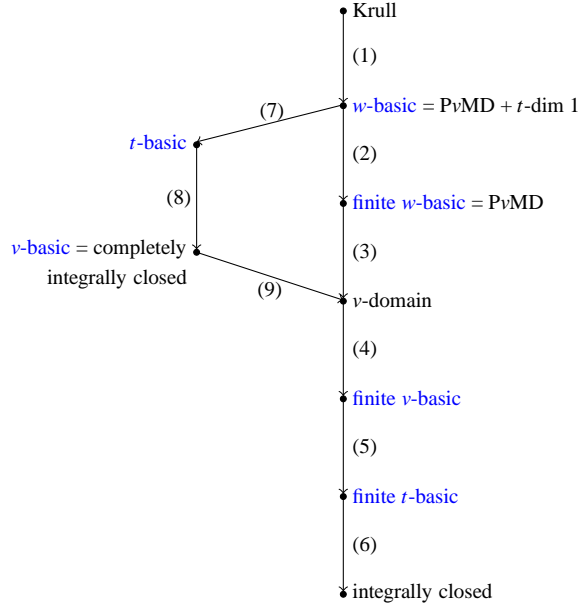

Of these implications, (1)-(3) and (9) are well known.  Implications (4)-(8) follow from Proposition~\ref{p:vdomain}, Proposition~\ref{p:vimpliest}, Lemma~\ref{l:starbasic}, Theorem~\ref{t:basic} (and the fact that $w=t$ in a P$v$MD), and Proposition~\ref{p:cic}, respectively.

Irreversibility of arrows (1)-(3) and (9) is again well known.  We do not know whether (5) is reversible.  The remainder of the paper is devoted to examples for (irreversibility of) the other implications.

\begin{example} Arrow (4) is irreversible.
\end{example}
\begin{proof} Let $k$ be a field and $X, Y, Z$ indeterminates over $k$. Let $T:=k((X))+M$ and $R:=k[[X]]+M$, where $M:=(Y,Z)k((X))[[Y, Z]]$. Let $A$ be an ideal of $R$.  Then $A$ is comparable to $M$.  Suppose $A \subseteq M$ and $A$ is not invertible.  If $AA^{-1} \supsetneq M$, then $AA^{-1}$ is principal, and hence $A$ is invertible, contrary to assumption. Hence $AA^{-1} \subseteq M$. We claim that $(AA^{-1})_v=M$.  To verify this, first recall that $M$ is divisorial in $R$.  Then, since $AA^{-1}$ is a trace ideal, that is, $(AA^{-1})^{-1}=(AA^{-1}:AA^{-1})$,  we have $(AA^{-1})^{-1} \subseteq (AA^{-1}T:AA^{-1}T)=T =M^{-1}$ (the first equality holding since $T$ is Noetherian and integrally closed).  This forces $(AA^{-1})^{-1}=M^{-1}$, whence $(AA^{-1})_v=M_v=M$, as claimed.  Now let $I$ be a finitely generated ideal of $R$ and $J$ a $v$-reduction of $I$, so that $(JI^n)_v=(I^{n+1})_v$ for some positive integer $n$.  We shall show that $J^{-1}=I^{-1}$ (and hence that $J_v=I_v$), and for this we may assume that $I$ is not invertible.  Suppose, by way of contradiction, that $IT(T:IT)=T$, i.e., that $IT$ is invertible in $T$.  Then, since $T$ is local, $IT$ is principal and, in fact, $IT=aT$ for some $a \in I$.  We then have $R \subseteq a^{-1}I \subseteq T$.  Then $k[[X]] \cong R/M \subseteq a^{-1}I/M \subseteq T/M \cong k((X))$, from which it follows that $a^{-1}I/M$ must be a cyclic $k[[X]]$-module.  However, this is easily seen to imply that $a^{-1}I$, hence $I$, is principal, the desired contradiction.  We therefore have $(T:IT)I \subseteq M$, whence $$(IM)^{-1}=(R:IM)=((R:M):I) = (T:I)=(M:I)\subseteq I^{-1}.$$ This immediately yields $I^{-1}=(IM)^{-1}$.

Now set $Q=I^n(I^n)^{-1}$. From above (setting $A=I^n$), we have $Q_v=M$.  Therefore, $$I^{-1}\subseteq J^{-1}\subseteq (JM)^{-1}=(JQ)^{-1}=(IQ)^{-1}=(IM)^{-1}=I^{-1},$$ which yields $J^{-1}=I^{-1}$, as desired.  Hence $R$ has the finite $v$-basic property.  Finally, again from above, we have $((y,z)(y,z)^{-1})_v = M$, so that $R$ is not a $v$-domain.

\end{proof}

\begin{example} Arrow (6) is irreversible.
\end{example}
\begin{proof} Let $k$ be a field and $X, Y$ indeterminates over $k$. Let $V=k(X)[[Y]]$ and $R=k+M$, where $M=Yk(X)[[Y]]$. Clearly, $R$ is an integrally closed domain. Of course, $M$ is divisorial in $R$.  Also, $(M^2)^{-1}=((R:M):M)=(V:M)=Y^{-1}V$, and so $(M^2)_v=(R:Y^{-1}V)=Y(R:V)=YM=M^2$, i.e., $M^2$ is also divisorial. We claim that $R$ does not have the finite $t$-basic ideal property. Indeed, let $W:=k+Xk$ and consider the finitely generated ideal $I$ of $R$ given by  $I=Y(W+M)$. We have $(k:W)=(0)$; otherwise, we have $0\not=f\in (k:W)$, and both $f$ and $fX\in k$, whence $X\in k$, a contradiction.   Therefore, $I^{-1}=Y^{-1}M$ and thus $I_{t}=I_{v}=YM^{-1}=M$. Now, let $J=YR$. Then $J_t=YR \subsetneq M=I_t$.  However, $$(JI)_{t}=(YI)_{t}=YI_{t}=YM=M^{2}=((I_{t})^{2})_t=(I^2)_t,$$ and so $R$ does not have the finite $t$-basic ideal property.
\end{proof}

\begin{example} \label{e:hoexample} Arrow (7) is irreversible.
\end{example}
\begin{proof} In \cite{ho} Heinzer and Ohm give an example of an essential domain that is not a P$v$MD. In that example, $k$ is a field, $y$, $z$, and $\{x_i\}_{i=1}^{\infty}$ are indeterminates over $k$, and $D=R \cap (\bigcap_{i=1}^{\infty} V_i)$, where $R=k(\{x_i\})[y,z]_{(y,z)k(\{x_i\})[y,z]}$ and $V_i$ is the rank-one discrete valuation ring on $k(\{x_j\}_{j=1}^{\infty},y,z)$ with $x_i,y,z$ all having value 1 and $x_j $ having value $0$ for $j \ne i$ (using the ``infimum'' valuation).  As further described in 
\cite[Example 2.1]{MZ}, we have $\Max(D)=\{M\} \cup \{P_i\}$, where $M$ is the contraction of $(y,z)R$ to $D$ and the $P_i$ are the centers of the maximal ideals of the $V_i$; moreover, $D_M=R$ and $V_i=D_{P_i}$.

It was pointed out in \cite[Example 1.7]{ghl} that each finitely generated ideal of $D$ is contained in almost all of the $V_i$.  If fact, one can say more.  Let $a$ be an element of $D$.  We may represent $a$ as a quotient $f/g$ with $f,g \in T:=k[\{x_i\},y,z]_{(y,z)k[\{x_i\},y,z]}$ and $g \notin (y,z)T$ (and hence $g \notin M$). Since $f$ and $g$ involve only finitely many $x_j$ and $g \notin M$, the sequence $\{v_i(a)\}$ must be eventually constant, where $v_i$ is  the valuation corresponding to $V_i$.  We denote this constant value by $w(a)$.  A similar statement holds for finitely generated ideals of $D$.

Let $K$ be a nonzero ideal of $D$. Then $$K_tD_{P_i} \supseteq KD_{P_i} = (KD_{P_i})_{t_{D_{P_i}}} = (K_tD_{P_i})_{t_{D_{P_i}}} \supseteq K_tD_{P_i},$$ whence $K_tD_{P_i}=KD_{P_i}$.

Now suppose that we have nonzero ideals $J \subseteq I$ of $D$ with $(JI^n)_t=(I^{n+1})_t$.  Let $a \in I$, and choose $a_0 \in I$ so that $w(a_0)$ is minimal.  Then $aa_0^n \in I^{n+1} \subseteq (JI^n)_t$, and so $aa_0^n \in (BA^n)_v$ for finitely generated ideals $B \subseteq J$ and $A \subseteq I$.  With the observation in the preceding paragraph, we then have $aa_0^n \in BA^nD_{P_i}$ for each $i$.  However, since $w(a_0) \le w(A)$, it must be the case that $w(a) \ge w(B)$; i.e., for some integer $k$, $a \in BD_{P_i}$ for all $i >k$.  Since the equality $(JI^n)_t=(I^{n+1})_t$ yields $JD_{P_i}=ID_{P_i}$ for each $i$, we may choose elements $b_j \in J$ for which $v_j(a)=v_j(b_j)$, $j=1, \ldots,k$.  With $B'=(B, b_1, \ldots, b_k)$, we then have $a \in B'D_{P_i}$ for each $i$.  This yields $a(B')^{-1} \subseteq \bigcap D_{P_i}$.

Next, we consider extensions to $D_M$.  From $(JI^n)_t=(I^{n+1})_t$, we obtain $(JI^nD_M)_{t_{D_M}}=(I^{n+1}D_M)_{t_{D_M}}$.  Since $D_M$ is a regular local ring, each nonzero ideal of $D_M$ is $t$-invertible, and we may cancel to obtain $(ID_M)_{t_{D_M}}=(JD_M)_{t_{D_M}}$.  There is a finitely generated subideal $B_1$ of $J$ with $B_1D_M=JD_M$.  We then have $$IB_1^{-1} \subseteq ID_M  B_1^{-1}D_M=ID_M (B_1D_M)^{-1}\subseteq (JD_M(JD_M)^{-1})_{t_{D_M}} \subseteq D_M.$$

Now let $B_2=B'+B_1$.  Then $a(B_2)^{-1} \subseteq D_M \cap \bigcap D_{P_i}=D$, whence $a \in (B_2)_v \subseteq J_t$.  It follows that $D$ has the $t$-basic property.  However, since $D$ is not a P$v$MD, $D$ cannot have the (finite) $w$-basic property.
\end{proof}

\begin{example} \label{e:entire} Arrow (8) is irreversible.
\end{example}
\begin{proof} Let $D$ denote the ring of entire functions.  It is well known that $D$ is a completely integrally closed Pr\"ufer domain of infinite Krull dimension.  Since $D$ is a Pr\"ufer domain, each nonzero ideal is a $t$-ideal.  The fact that $\dim D = \infty$ then yields that $D$ does not have the ($t$-) basic property by \cite[Theorem 10]{H2}.
\end{proof}

\end{document}